\documentclass[a4paper,11pt, reqno]{amsart}
\usepackage{amsmath,amssymb,amsfonts}
\usepackage{epsfig,color}
\usepackage{mathrsfs}
\usepackage{hyperref}
\usepackage{dsfont}
\newtheorem{defn}{Definition}[section]
\newtheorem{thm}{Theorem}[section]
\newtheorem{prop}{Proposition}[section]

\usepackage{tikz}

\usepackage{color}

\theoremstyle{definition}
\newtheorem{rmk}{Remark}[section]

\newcommand{\X}{\mathbf{X}}
\newcommand{\Y}{\mathbf{Y}}
\newcommand{\TX}{\mathbf{TX}}
\newcommand{\NN}{\mathbf{N}}
\newcommand{\T}{\mathbf{T}}
\newcommand{\SSw}{\mathbf{S}}
\newcommand{\ph}{\varphi}
\newcommand{\Z}{\mathbf{Z}}

\def\N{{\rm I\kern-0.16em N}}
\def\R{{\rm I\kern-0.16em R}}
\def\E{{\rm I\kern-0.16em E}}
\def\P{{\rm I\kern-0.16em P}}
\def\F{{\rm I\kern-0.16em F}}
\def\B{{\rm I\kern-0.16em B}}
\def\C{{\rm I\kern-0.46em C}}
\def\G{{\rm I\kern-0.50em G}}

\numberwithin{equation}{section}
\font\eka=cmex10
\usepackage{ae}
\def\ind{\mathrel{\hbox{\rlap{%
\hbox to 7.5pt{\hrulefill}}\raise6.6pt\hbox{\eka\char'167}}}}
\begin{document}

\title[Classical and free Fourth Moment Theorems]{Classical and free Fourth Moment Theorems: universality and thresholds}

\maketitle

\begin{center}
\authors{I. Nourdin, G. Peccati, G. Poly, R. Simone}
\end{center}

\begin{abstract}
Let $X$ be a centered random variable with unit variance, zero third moment, and such that $\E[X^4] \ge 3$.
Let $\{F_n : n\geq 1\}$ denote a normalized sequence of homogeneous sums of fixed degree $d\geq 2$, built from independent copies of $X$.
Under these minimal conditions, we prove that $F_n$ converges in distribution to a standard Gaussian
random variable if and only if the corresponding sequence of fourth moments converges to $3$.
The statement is then extended ({\it mutatis mutandis}) to the free probability setting. We shall also discuss the optimality of our conditions in terms of explicit thresholds, as well as establish several connections with the so-called {\it universality phenomenon} of probability theory. Both in the classical and free probability frameworks, our results extend and unify previous {\it Fourth Moment Theorems} for Gaussian and semicircular approximations. Our techniques are based on a fine combinatorial analysis of higher moments for homogeneous sums.
\end{abstract}

\section{Introduction}

\subsection{Overview}

The \textit{Fourth Moment Phenomenon} is a collection of mathematical statements, yielding that for many sequences of non-linear functionals of random fields (both in a commutative and non-commutative setting), a Central Limit Theorem is simply implied by the convergence of the corresponding sequence of fourth moments towards a prescribed target. Statements of this type represent a drastic simplification of the classical {\it method of moments and cumulants}, and are customarily called \textit{Fourth Moment Theorems}. As witnessed by the web resource \cite{web}, fourth moment results have recently gained enormous momentum within the probabilistic literature, in particular in view of their many applications, e.g. to Gaussian analysis, stochastic geometry and free probability (see \cite[Chapter 5]{NPbook}, \cite{Peccati_Lachieze} and \cite{NourdinPeccatiSpeicher} for some seminal references on the subject, as well as Section 1.3 for a more detailed discussion of the existing literature).

\medskip

The aim of the present paper is to fully explore the Fourth Moment Phenomenon in the framework of homogeneous sums in independent or freely independent random variables, by accomplishing the following tasks: {\bf (i)} to provide general simple sufficient conditions for Fourth Moment Theorems to hold, and {\bf (ii)} to link such Fourth Moment Theorems to {\it universality statements}, in the same spirit as the recent references \cite{NourdinDeya, NourdinPeccatiReinert, Peccati2}.

\subsection{Goals} We will now provide an informal presentation of our principal objectives: the reader is referred to Section 2 for definitions and precise assumptions. In what follows, we shall denote by $N$ and $S$, respectively, a random variable having a standard normal $\mathscr{N}(0,1)$ distribution, and a free random variable having the standard semicircular $\mathcal{S}(0,1)$ distribution. According to the usual convention, any statement implying convergence in law to $N$ or $S$ is called a `Central Limit Theorem' (CLT). As anticipated, the main focus of our paper is on sequences of {\it homogeneous sums} of the form
$$ Q_{\X}(f_n) = \sum_{i_1,\dots,i_d=1}^n f_n(i_1,\dots,i_d) X_{i_1}\cdots X_{i_d}, \quad n\geq 1,$$
where: (a) $d\geq 2$, (b) each mapping $f_n : [n]^d\to \R$ is symmetric and vanishes on diagonals (see Definition \ref{Admissible}), and (c) ${\bf X} = \{X_i : i\geq 1\}$ is either a collection of {\it independent} copies of a classical centered real-valued random variable with finite fourth moment (defined on some probability space $(\Omega, \mathcal{F}, \P)$), or a family of {\it freely independent} copies of a non-commutative centered random variable (defined on a suitable free probability space $(\mathcal{A}, \varphi)$). In both instances at Point (c), $Q_{\X}(f_n)$ is a centered random variable.

\medskip

Our principal aim is to tackle the following questions {\bf (I)}--{\bf (III)}:
\begin{enumerate}
\item[\bf (I)] Assume that ${\bf X}$ is defined on a classical probability space, and that the kernels $f_n$ are normalized in such a way that $Q_{\X}(f_n)$ has unit variance.  Under which conditions on the law of $X_1$ (remember that the $X_i$'s are identically distributed), the asymptotic condition (as $n\to\infty$)
\begin{equation}\label{e:1}
\E[Q_{\X}(f_n)^4] \longrightarrow 3\, (= \E[N^4]),
\end{equation}
is systematically equivalent to the convergence
\begin{equation}\label{e:2}
Q_{\X}(f_n) \xrightarrow{\text{\rm Law}} N\,\, ?
\end{equation}

\item[\bf (II)] Assume that ${\bf X}$ is defined on a free probability space, and that the kernels $f_n$ are chosen in such a way that the variance of each $Q_{\X}(f_n)$ equals one.  Under which conditions on the law of $X_1$ (again, the $X_i$'s are identically distributed) the asymptotic condition (as $n\to\infty$)
\begin{equation}\label{e:3}
\varphi(Q_{\X}(f_n)^4) \longrightarrow 2\, (= \varphi(S^4)),
\end{equation}
is systematically equivalent to the convergence
\begin{equation}\label{e:4}
Q_{\X}(f_n) \xrightarrow{\text{\rm Law}} S\,\, ?
\end{equation}

\item[\bf (III)] Are the convergence results at Points {\bf (I)} and {\bf (II)} {\it universal}? That is, assume that either \eqref{e:2} or \eqref{e:4} is verified: is it true that {\bf X} can be replaced by any other sequence of (freely) independent variables ${\bf Y} = \{Y_i :i\geq 1\}$, and still obtain that the sequence $Q_{\bf Y}(f_n)$, $n\geq 1$, verifies a CLT?
\end{enumerate}

Our main achievements are stated in Theorem \ref{superTeo1} and Theorem \ref{superTeo2}, where we provide surprisingly simple answers to questions {\bf (I)}--{\bf (III)}. Indeed, in Theorem \ref{superTeo1}, dealing with the classical case, it is proved that the conditions $\E[X^3_1]=0$ and $\E[X_1^4]\geq 3$ imply a positive answer both to Questions {\bf (I)} and {\bf (III)}. The free counterpart of this finding appears in Theorem \ref{superTeo2}, where it is proved that Questions {\bf (I)} and {\bf (III)} admit a positive answer whenever $\varphi(X_1^4)\geq 2$.

\medskip

Our proofs rely on a novel combinatorial representation for the fourth cumulant of $Q_{\X}(f_n)$, both in the classical and free case --- see, respectively, Proposition \ref{formulaGauss} and Proposition \ref{formula}. We believe that these results are of independent interest: in particular, they provide a new powerful representation of the fourth cumulant of a homogeneous sum in terms of `nested' cumulants of sums of lower orders, that perfectly complements the usual moments/cumulants relations for non-linear functionals of random fields (see e.g. \cite{Speicher, pt}).

\medskip

As indicated by the title, the last section of the present paper partially addresses the problem of thresholds, studying in particular the optimality of the conditions provided in Theorems \ref{superTeo1} and \ref{superTeo2}.

\subsection{Some history}

We now list some references that are relevant for the present paper.

\begin{itemize}
\item[--] In the classical probability setting, the first study of Fourth Moment Theorems for homogeneous sums (or, more generally, for degenerate $U$-statistics) appears in the landmark papers by de Jong \cite{deJong1, deJong2}, where it is proved that, if the kernels $\{f_n\}$ verify a Lindeberg-type condition and $X_1$ has finite fourth moment, then condition \eqref{e:1} is sufficient for having the CLT \eqref{e:2}. Our results will show, in particular, that such Lindeberg type condition can be dropped whenever $X_1$ has a fourth moment that is greater than 3. Recent developments around the theorems by de Jong appear in \cite{the, peth}.

\item[--] A crucial reference for our analysis is \cite{Nualart}, where the authors proved that the implication \eqref{e:1} $\rightarrow$ \eqref{e:2} holds systematically (without any restriction on the kernels $\{f_n\}$) whenever ${\bf X}$ is composed of independent and identically distributed $\mathscr{N}(0,1)$ random variables. It is important to notice that the result continues to hold when one replaces homogeneous sums with a sequence living in a fixed {\it Wiener chaos} of a general Gaussian field. Reference \cite{Nualart} has strongly motivated the introduction of the so-called {\it Malliavin-Stein method} (first developed in \cite{NP_ptrf}), that is by now one of the staples of the asymptotic analysis of functionals of Gaussian fields. See \cite{web}, as well as the monograph \cite{NPbook} for an introduction to this subject, see \cite{NourdinPeccatiReveillac, PeccatiTudor} for multidimensional generalisations, and \cite{NourdinPeccatiSwan} for extensions to an information-theoretical setting. The reader can consult \cite{Guillaume} for an alternative simple proof of the results of \cite{Nualart} and \cite{EhsanGuillaume} for an extension of the Fourth Moment Theorem to higher moments.

\item[--] Reference \cite{Peccati1} shows that the implication \eqref{e:1} $\rightarrow$ \eqref{e:2} also holds when ${\bf X}$ is composed of independent Poisson random variables. See also \cite{the, Peccati_Lachieze, Peccati2, s} for some partial results involving sequence of random variables living inside a fixed Poisson chaos. Fourth moment-type statements for infinitely divisible laws can be found in \cite{Arizmendi2}.

\item[--] In the free probability setting, the analysis of the \textit{Fourth Moment Phenomenon} for non-linear functionals of a free Brownian motion was started in \cite{NourdinPeccatiSpeicher}, where the authors proved a non-commutative counterpart to the findings of \cite{Nualart}. Extensions are provided in \cite{qbrownian} for multiple integrals with respect to the $q$-Brownian motion and in \cite{Solesne} for the free Poisson Chaos. See moreover \cite{NouSpeiPec} for a discussion about the multidimensional case.

\item[--] As anticipated, one major feature of fourth moment results is that they are often associated with \textit{universality} statements. In the classical case, the most relevant reference is arguably \cite{NourdinPeccati4}, where the authors proved that homogeneous sums involving independent Gaussian random variables are universal in the sense detailed in Question {\bf (III)} of the previous section. This was accomplished by exploiting the invariance principle established in \cite{Mossel}. Similar results in the Poisson case are discussed in \cite{Peccati1, Peccati2}.

\item[--] The non-commutative version of \cite{Mossel} can be found in \cite{NourdinDeya}, yielding in turn the universality of the semicircular law for symmetric homogeneous sums. These results have been recently extended in \cite{Simone} to the multidimensional setting.
\end{itemize}

\subsection{Plan} The rest of the paper is organized as follows. Section 2 contains the formal statement of our main results. Section 3 contains the statements and proofs of the announced new formulae for fourth moments of homogeneous sums. Section 4 is devoted to proofs, whereas Section 5 deals with partial thresholds characterizations.

\section{Framework and main results}

\subsection{Admissible kernels}

For every $n\in\N$, we shall use the notation $[n] := \{1,\dots,n\}$. The following terminology is used throughout the paper.

\begin{defn}
\label{Admissible}\rm{
For a given degree $d\geq 2$ and some integer $n\geq 1$, a mapping $f:[n]^d\to\R$ is said to be an {\it admissible kernel} if the following three properties are satisfied:
\begin{enumerate}
\item[($\alpha$)] $f(i_1,\cdots,i_d)=0$ if the vector $(i_1,\cdots,i_d)$ has at least one repetition, that is, if  $i_j=i_k$ for some $k\neq j$;
\item[($\beta$)] $f(i_1,\dots,i_d)=f(i_{\sigma(1)},\dots,i_{\sigma(d)})$ for any permutation $\sigma$ of $\{1,\dots,d\}$ and any $(i_1,\dots,i_d)\in [n]^d$;
\item[($\gamma$)] the normalization $$d!\sum\limits_{i_1,\dots,i_d =1}^n f(i_1,\dots,i_d)^2=1$$ holds.
\end{enumerate}}
\end{defn}

For instance, in the case $d=2$, the kernel  $f(i,j) =  {\bf 1}_{\{i\neq j\}}/\sqrt{2n(n-1)}$ is admissible. Note that, given a mapping $f:[n]^d\to\R$ verifying $(\alpha)$, it is always possible to generate an admissible kernel $\tilde{f}$ by first symmetrizing $f$ (in order to meet requirement $(\beta)$), and then by properly renormalizing it (in order to meet requirement $(\gamma)$).

\subsection{Main results in the commutative case} Every random object considered in the present section lives on a suitable common probability space $(\Omega,\mathcal{F},\P)$. In what follows, we consider a random
variable $X$ satisfying the following Assumption {\bf (A)}:

\smallskip

\begin{itemize}
\item[\bf (A)] $X$ has finite fourth moment, and moreover $\E[X]=0$, $\E[X^2]=1$, and $\E[X^3]=0$.
\end{itemize}

\smallskip

We will denote by $\X=\{X_i : i\geq 1\}$ a sequence composed of independent copies of  $X$. For any admissible kernel $f:[n]^d\to\R$, we define $Q_{\X}(f)$ to be the random variable
\begin{eqnarray}\label{e:alv}
Q_{\X}(f) &=&\sum_{i_1,\dots,i_d =1}^n f(i_1,\dots,i_d) X_{i_1}\cdots X_{i_d}.\label{F}
\end{eqnarray}
In view of our assumption on ${\bf X}$ and $f$, it is straightforward to check that $\E[Q_{\X}(f) ]=0$ and $\E[Q_{\X}(f)^2]=1$.

\smallskip

Finally, we write $\mathscr{N}(0,1)$ to indicate the standard Gaussian distribution, that is, $\mathscr{N}(0,1)$ is  the probability distribution on $(\R, \mathscr{B}(\R))$ with density $x\mapsto \frac{1}{\sqrt{2\pi}}e^{-x^2/2}$. In agreement with the above conventions, given $N\sim \mathscr{N}(0,1)$, we shall denote by ${\bf N} = \{N_i : i\geq 1\}$ a sequence of independent copies of $N$.
\smallskip

The following definition plays a pivotal role in our work.

\begin{defn}\label{defcom}{\rm
Fix $d\geq 2$, let $X$ satisfy Assumption {\bf (A)} and let $N\sim \mathscr{N}(0,1)$; define the sequences ${\bf X}$ and ${\bf N}$ as above. Consider a sequence $f_n:[n]^d\to\R$, $n\geq 1$, of
admissible kernels, as well as the following asymptotic relations, for $n\to\infty$:
\begin{enumerate}
\item[(i)] $Q_{\X}(f_n) \xrightarrow{\text{\rm Law}}\mathscr{N}(0,1)$;
\item[(ii)]  $Q_{\bf N}(f_n) \xrightarrow{\text{\rm Law}}\mathscr{N}(0,1)$;
\item[(iii)]  $\E[Q_{\X}(f_n)^4]\to 3$.
\end{enumerate}
Then:

{\rm 1.} We say that $X$ satisfies the {\it Commutative Fourth Moment Theorem} of order $d$ (in short: ${CFMT}_d$) if, for any sequence $f_n:[n]^d\to\R$ of admissible kernels, {\rm (iii)} implies {\rm (i)} as $n\to\infty$.

{\rm 2.} We say that $X$ is {\it Universal at the order} $d$ (in short: $U_d$) if, for any sequence $f_n:[n]^d\to\R$ of admissible kernels, {\rm (i)} implies {\rm (ii)}.}
\end{defn}

The use of the term `universal' adopted in the previous definition may seem slightly obscure. To justify our terminology (and for future reference) we recall the following invariance principle from \cite{NourdinPeccati4}, implying in particular that, if a sequence of admissible kernels verifies (ii), then (i) is automatically satisfied for every sequence ${\bf X}$ of i.i.d. centered normalized random variables.
\begin{thm}[See \cite{NourdinPeccati4}]\label{inv}
Fix $d\geq 2$, let $f_n:[n]^d\to\R$ ($n\geq 1$) be a sequence of admissible kernels and let $N\sim\mathscr{N}(0,1)$.
Assume that $Q_{\bf N}(f_n) \xrightarrow{\text{\rm Law}}\mathscr{N}(0,1)$, as $n\to\infty$.
Then, $Q_{\X}(f_n) \xrightarrow{\text{\rm Law}}\mathscr{N}(0,1)$
for each random variable $X$ satisfying $\E[X]=0$ and $\E[X^2]=1$. Moreover, as $n\to\infty$,
\[
\max_{1\leq i \leq n}\sum_{i_2,\ldots,i_d=1}^n f_n(i ,i_2,\ldots,i_d)^2 \to 0.
\]
\end{thm}

The following result is a direct consequence of the main results of \cite{Nualart}.

\begin{thm}[See \cite{Nualart}]\label{np}
For any $d\geq 2$, the random variable $N\sim \mathscr{N}(0,1)$
satisfies the ${CFMT}_d$.
\end{thm}

\medskip

The subsequent Theorem \ref{superTeo1} is one of the main achievements of the present paper. It provides a strikingly simple sufficient condition under which $X$ is both $U_d$ and satisfies a ${CFMT}_d$.
The proof (that is detailed in Section 4) relies on Theorem \ref{np}, as well as on the combinatorial formula for the fourth moment of $Q_{\X}(f)$ provided in Section 3 (see formula \eqref{formulaGauss}).

\begin{thm}
\label{superTeo1}
Fix a degree $d\geq 2$ and assume that $X$ satisfies Assumption {\bf (A)} and $\E[X^4]\ge 3$. Then, $X$ is $U_d$ and satisfies the ${CFMT}_d$.
\end{thm}

In the next subsection, we shall present a non-commutative counterpart to Theorem \ref{superTeo1}.

\subsection{Main results in the non-commutative case}

We shortly recall few basic facts about non-commutative (or free) probability theory that will be useful for our discussion: we refer the reader to the monograph \cite{Speicher} for any unexplained definition or result in this domain. Let $(\mathcal{A},\varphi)$ be a non-commutative (or free) probability space (namely, $\mathcal{A}$ is a unital $\ast$-algebra, and $\varphi$ is a faithful positive state, customarily called a {\it trace}). As usual, self-adjoint bounded elements in $\mathcal{A}$ will be referred to as {\it random variables}. Let $\rho(Y) \in [0,\infty)$ denote the \textit{spectral radius} of a given random variable $Y$,
namely $\rho(Y)= \lim\limits_{k \rightarrow \infty} \varphi(Y^{2k})^{\frac{1}{2k}}$. The existence of a real measure $\mu_{Y}$ with compact support contained in $[-\rho(Y),\rho(Y)]$  such that:
$$  \varphi(Y^{k}) = \int_{\mathbb{R}}y^{k}\mu_{Y}(dy),$$
is established e.g. in \cite[Theorem 2.5.8]{Tao} and \cite[Proposition 3.13]{Speicher}. Such a measure $\mu_Y$ is called the \textit{law}, or the \textit{distribution} of $Y$. Being compactly supported, $\mu_Y$ is completely determined by the sequence of its moments $\{\varphi(Y^{m}): m \in \mathbb{N}\}$.
Given free probability spaces $(\mathcal{A}_n, \varphi_n)$, $(\mathcal{A},\varphi)$, and a sequence of random variables $a_n \in \mathcal{A}_n$ and $a \in \mathcal{A}$, we say that $\{a_n\}$ {\it converges in distribution} (or {\it in law}) towards $a$ if:
$$ \lim_n \varphi_n(a_n^k) = \varphi(a^k) \qquad \forall k \geq 1.$$
If the above asymptotic relations are verified for every $k$, we write $a_n \xrightarrow{\text{\rm Law}}a$. In the non-commutative setting, the role of independence is played by the notion of {\it freeness}: the unital subalgebras $\mathcal{A}_1, \dots, \mathcal{A}_n$ of $\mathcal{A}$ are said to be \textit{freely independent} if, for every $k \geq 1$, for every choice of integers $i_1,\dots, i_k$ with $i_j \neq i_{j+1}$, and random variables $a_{i_j} \in \mathcal{A}_j$, we have $\varphi(a_{i_1}a_{i_2}\cdots a_{i_k}) = 0$. The random variables $a_1,\dots, a_n$ are said to be freely independent if the (unital) subalgebras they generate are freely independent.

Let us consider a random variable $Y$ in $(\mathcal{A},\varphi)$ satisfying the following Assumption {\bf (B)}:
\begin{itemize}
\item [\bf (B)]$\varphi(Y)=0$ and $\varphi(Y^2)=1$.
\end{itemize}
Note that, in the non-commutative setting, $Y$ has all moments by definition; moreover, in contrast to Assumption {\bf (A)} used in the commutative case, here we will not need the vanishing of the third moment $\varphi(Y^3)$.

Similarly to the conventions of the previous subsection, we will denote by $\Y=\{Y_i : i\geq 1\}$ a sequence composed of freely independent copies of $Y$, that we can assume defined on $(\mathcal{A},\varphi)$. When $f:[n]^d\to\R$ is an admissible kernel, we define $Q_{\Y}(f)$ as the random variable
\begin{eqnarray}
Q_{\Y}(f) &=&\sum_{i_1,\dots,i_d =1}^n f(i_1,\dots,i_d) Y_{i_1}\cdots Y_{i_d}.\label{Fnoncom}
\end{eqnarray}
In view of our assumptions on ${\bf Y}$ and $f$, it is straightforward to check that $\varphi(Q_{\Y}(f) )=0$ and $\varphi(Q_{\Y}(f)^2)=d!^{-1}$. We shall also use the following consequence of \cite[formula (2.3)]{NourdinPeccatiSpeicher}: for every $d \geq 2$ and every admissible kernel $f:[n]^d\to\R$:
\small{
\begin{eqnarray}\notag
&& \varphi\left( Q_{\SSw}(f) ^4\right) = 2\left( \sum_{j_1,...,j_d=1}^n f(j_1,...,j_d)^2  \right)^2 \; + \sum_{s=1}^{d-1} \sum_{j_1,...,j_{d-s}=1}^n   \sum_{k_1,...,k_{d-s}=1}^n\notag \\
 && \quad \quad \quad \label{e:knps}\left(\sum_{a_1,...,a_s=1}^n f(j_1,...,j_{m-s}, a_1,...,a_s)f(k_1,...,k_{m-s}, a_1,...,a_s)  \right)^2\; .\notag
\end{eqnarray}
}
\normalsize
\smallskip

Finally, $\mathcal{S}(0,1)$ will denote the standard semicircular distribution, namely the distribution on $(\R, \mathscr{B}(\R))$ with density $x\mapsto \frac{1}{2\pi}\sqrt{4-x^2}$ and support given by the interval $[-2,2]$. It is easily checked that, if $S \sim \mathcal{S}(0,1)$, then $\varphi(S)=0$, $\varphi(S^2) = 1$, $\varphi(S^3)=0$, and $\varphi(S^4)=2$. Given $S\sim \mathcal{S}(0,1)$ and following the previous conventions, the bold letter ${\bf S}$ will stand for a sequence of freely independent copies of $S$.

The following definition is the non-commutative counterpart of Definition \ref{defcom}. Note the additional factor $d!$ appearing in many of the relations below: this is due to the fact that, in the non-commutative setting, the variance of a homogeneous sum based on an admissible kernel equals $d!^{-1}$, whereas the variance of the standard semicircular distribution is 1.

\begin{defn}\label{defnoncom}{\rm
Fix $d\geq 2$, let $Y$ satisfy Assumption {\bf (B)} and let $S\sim \mathcal{S}(0,1)$; define ${\bf Y}$ and ${\bf S}$ as above. Let $f_n:[n]^d\to\R$, $n\geq 1$, be a sequence of admissible kernels, and consider the following asymptotic relations (as $n\to\infty$):
\begin{enumerate}
\item[(i)'] $\sqrt{d!} \,Q_{\Y}(f_n) \xrightarrow{\text{\rm Law}}\mathcal{S}(0,1)$;
\item[(ii)']  $\sqrt{d!}\,Q_{\bf S}(f_n) \xrightarrow{\text{\rm Law}}\mathcal{S}(0,1)$;
\item[(iii)']  $d!^2\varphi(Q_{\Y}(f_n)^4)\to 2$.
\end{enumerate}
Then:

1. We say that $Y$ satisfies the {\it Non-Commutative Fourth Moment Theorem} of order $d$ (in short: ${NCFMT}_d$) if, for any sequence $f_n:[n]^d\to\R$ of admissible kernels, (iii)' implies (i)' as $n\to\infty$.

2. We say that $X$ is {\it Freely Universal} at the order $d$ (in short: $FU_d$) if, for any sequence $f_n:[n]^d\to\R$ of admissible kernels,  (i)' implies (ii)'.}
\end{defn}

The next statement is a free analogue of Theorem \ref{inv} (see \cite{NourdinDeya}).

\begin{thm}[See \cite{NourdinDeya}]\label{invnoncom}
Let $f_n:[n]^d\to\R$ be a sequence of admissible kernels and let $S\sim\mathcal{S}(0,1)$.
Assume that $\sqrt{d!}Q_{\bf S}(f_n) \xrightarrow{\text{\rm Law}}\mathcal{S}(0,1)$ as $n\to\infty$.
Then, $\sqrt{d!}Q_{\Y}(f_n) \xrightarrow{\text{\rm Law}}\mathcal{S}(0,1)$
for each random variable $Y$ satisfying {\bf (B)}.
\end{thm}

According to the following result from \cite{NourdinPeccatiSpeicher}, an example of distribution satisfying the $NCFMT_d$ for any $d\geq 2$ is the $\mathcal{S}(0,1)$ distribution itself.

\begin{thm}[See \cite{NourdinPeccatiSpeicher}]\label{knps}
For any $d\geq 2$, the standard semicircular random variable $S\sim \mathcal{S}(0,1)$
satisfies the ${NCFMT}_d$.
\end{thm}

The subsequent Theorem \ref{superTeo2} is our main achievement in the non-commutative case. It yields a sufficient condition under which $Y$ is both $FU_d$ and satisfies a ${NCFMT}_d$. The proof (see Section 4) relies on Theorem \ref{knps} and on the non-commutative combinatorial relation \eqref{formula}, whose proof is detailed in the forthcoming Section 3.

\begin{thm}
\label{superTeo2}
Fix a degree $d\geq 2$ and, in addition to Assumption {\bf (B)}, assume that  $\varphi(Y^4) \ge 2$. Then, $Y$ is $FU_d$  and satisfies  the $NCFMT_d$.
\end{thm}

\medskip

The next section provides the combinatorial formulae that are needed in the proofs of Theorem \ref{superTeo1} and Theorem \ref{superTeo2}.

\section{Combinatorial formula\!e for fourth moments}

\subsection{Preliminaries}
Our main aim in this section is to prove formulae \eqref{formulaGauss} and \eqref{formula}, yielding new combinatorial expressions for the fourth moment of homogeneous sums, respectively in the classical and free cases. Although our main motivation for deriving these relations comes from the fact that they are crucial elements in the proofs of Theorem \ref{superTeo1} and Theorem \ref{superTeo2}, we believe that these relations have an intrinsic interest, as they provide a new approach to cumulant computations for non-linear functionals of a given sequence of independent or freely independent random variables. We refer the reader e.g. to \cite{Speicher, pt} for any unexplained combinatorial definition or result (respectively, in the non-commutative and commutative cases).

\smallskip

We denote by $\mathcal{P}([n])$ and $\mathcal{NC}([n])$, respectively, the lattice of all partitions and the lattice of  all non-crossing partitions of the set $[n]$. Also, $\mathcal{P}_2([n])$ and $\mathcal{NC}_2([n])$ will stand, respectively, for the set of all pairings and all non-crossing pairings of $[n]$. The definitions of $\mathcal{P}(A)$ and $\mathcal{P}_2(A)$ for a generic finite set $A$ are analogous.
Recall that a partition $\pi$ of the set $[n]$ is said to be \textit{non-crossing} if, whenever there exist integers $i < j < k < l$, with $i\sim_{\pi} k$, $j \sim_{\pi} l$, then $j \sim_{\pi} k$ (here, $i \sim_{\pi} j$ means that $i$ and $j$ belong to the same block of $\pi$).
\smallskip

In this section, the following interval partition in $\mathcal{P}([4d])$ ($d\geq 2$) will play a crucial role:
$$\pi^{\star} = \{\{1,\cdots, d\} , \{ d+1,\cdots, 2d\}, \{ 2d+1,\cdots, 3d\} , \{ 3d+1,\cdots, 4d\}\},$$
that is: $\pi^\star$ is the partition of $[4d]$ composed of consecutive intervals of exact length $d$. Given a partition $\sigma$ of $[4d]$,  and according to the usual lattice notation, we will write $\sigma \wedge \pi^{\star}= \hat{0}$ to indicate that, in each block of $\sigma$, there is at most one element from every block of $\pi^\star$ (this property is usually referred to by saying that ``$\sigma$ \textit{respects} $\pi^\star$'').

\smallskip

Finally, we will use the symbol $\mathcal{P}^\star_{2,4}([4d])$ (resp.
$\mathcal{NC}^\star_{2,4}([4d])$) to denote the set of those partitions (resp. non-crossing partitions) of $[4d]$ that respect $\pi^\star$ and that contain only blocks of sizes $2$ and $4$.
Analogously, the symbol $\mathcal{P}^\star_{2}([4d])$ (resp.
$\mathcal{NC}^\star_{2}([4d])$) is used to denote the set of those pairings (resp. non-crossing pairings) belonging to $\mathcal{P}^\star_{2,4}([4d])$ (resp.
$\mathcal{NC}^\star_{2,4}([4d])$).

\subsection{Fourth moment formulae in the commutative case}

For a generic random variable $X$ with moments of all orders, we shall denote by $\chi_j(X)$, $j=1,2,...$ the (classical) $j$th cumulant of $X$ (see \cite[Chapter 3]{pt}).

For $\pi \in \mathcal{P}([n])$, the generalized \textit{cumulant} $\chi_{\pi}(X)$ of a random variable $X$ is defined as the multiplicative mapping $X\mapsto \chi_\pi(X)= \prod_{b\in \pi} \chi_{|b|}(X)$. Such a mapping satisfies the formula:
$$ \E[X^n] = \sum_{\pi \in \mathcal{P}([n])} \chi_{\pi}(X),$$
which is equivalent to
$$\chi_n(X) = \sum_{\pi \in \mathcal{P}([n])}  \mu(\pi, \hat{1}) \prod_{b \in \pi} \E[X^{|b|}],$$
with $\mu(\pi,\hat{1})$ denoting the M\"{o}bius function on the interval $[\pi, \hat{1}]$.
We have therefore that $\chi_1(X) = \E[X]$, $\chi_2(X) = \E[X^2]- \E[X]^2$ (the variance of $X$); if $X$ is centered, then $\chi_4(X) = \E[X^4]-3\E[X^2]^2$.

Given a vector of random variables of the type $(X_{i_1},...,X_{i_{n}})$, (with possible repetitions), for every partition $\sigma = \{b_1,...,b_k\} \in \mathcal{P}([n])$, we define the generalized joint cumulant
$$
\chi_\sigma(X_{i_1},...,X_{i_{n}}) = \prod_{j=1}^k \chi(X_{i_a} : a\in b_j),
$$
where $\chi(X_{i_a} : a\in b_j)$ is the joint cumulant of the random variables composing the vector $(X_{i_a} : a\in b_j)$, as defined e.g. in \cite[Chapter 3]{pt}, satisfying
$$ \mathbb{E}[X_{i_1}\cdots X_{i_n}] = \sum_{\sigma \in \mathcal{P}([n])}\chi_\sigma(X_{i_1},...,X_{i_{n}}). $$
If the $X_j$'s are independent,
$$
\chi(X_{i_a} : a\in b_j) =
\begin{cases}
\chi_{|b_j|} (X_i)  &  \text{ if } i_a = i \text{ for every } a\in b_j \\
0 & \textit{otherwise}.
\end{cases}
$$

The next statement contains the announced combinatorial representation of the fourth cumulant of a homogeneous sum in the classical case.
\begin{prop}\label{propgauss}
Let $d\geq 2$ and let $f:[n]^d\to\R$ be an admissible kernel. For $m=1,\dots, d$ and for fixed $j_1,...,j_m\in [n]$, denote by $Q_\NN(f(j_1,\dots,j_m,\cdot))$ the homogeneous sum of order $d-m$ based on a sequence of i.i.d. standard Gaussian random variables ${\bf N}=\{N_i : i\geq 1\}$ and on the kernel
$$ (i_1,\dots,i_{d-m}) \mapsto f(j_1,\dots,j_m,i_1,\dots,i_{d-m}),$$
namely,
$$
Q_{\NN}(f(j_1,\dots,j_m,\cdot)) = \sum\limits_{i_1,\dots,i_{d-m}=1}^n f(j_1,\dots,j_m, i_1,\dots,i_{d-m})N_{i_{1}}\cdots N_{i_{d-m}}.
$$
If $X$ satisfies Assumption {\bf (A)}, then
\begin{eqnarray}
\chi_4(Q_\X(f)) \label{formulaGauss}
&=& \chi_4(Q_\NN(f)) \\
&&+ \sum_{m=1}^{d} \binom{d}{m}^4 m!^{4} \chi_4(X)^m \sum_{j_1,\dots,j_m=1}^n \E[Q_\NN(f(j_1,\dots,j_m,\cdot))^4].\notag
\end{eqnarray}
\end{prop}

\begin{proof}
Throughout the proof, one may and will identify each partition $\sigma \in \mathcal{P}^\star_{2,4} ([4d])$ with the triplet $(m, {\bf u}, \tau)$ (and we will write $\sigma(m,{\bf u},\tau)$ instead of $\sigma$ accordingly), where:
\begin{enumerate}
\item[(a)] $m\in\{0,1,...,d\}$ indicates the number of 4-blocks of $\sigma$,
\item[(b)] the collection of four {\it ordered} $m$-tuples with no repetitions
$$
{\bf u} = \{(u_1^1,...,u_m^1), ..., (u_1^4,...,u_m^4)\},
$$
is such that, for every $j=1,...,m$ and every $i=1,2,3,4$, the symbol $u_j^i$ stands for the element of the $i$th block of $\pi^{\star}$ that is contained in the $j$th 4-block of $\sigma$,
\item[(c)] $\tau \in \mathcal{P}_2(A({\bf u}))$ is a uniquely defined pairing respecting the restriction of $\pi^{\star}$ to $A({\bf u}) := [4d] \backslash \{u_1^1,...,u_m^1, ..., u_1^4,...,u_m^4\}.$
\end{enumerate}
In order to simplify the forthcoming exposition, in what follows we stipulate that the $2d-2m$ blocks $b_1,...,b_{2d-2m}$ of $\tau$ are ordered according to the increasing size of the smallest element. For instance, if $d=3$, $m=1$, and\linebreak $\sigma= \{\{1,4,7,10\}, \{2,12\}, \{6,8\}, \{5,9\}, \{3,11\}\}$, then one has that \linebreak $A({\bf u}) = \{2,3,5,6,8,9,11,12\}$, and $$\tau = \{b_1,b_2, b_3,b_4\} =  \{\{2,12\}, \{3,11\}, \{5,9\}, \{6,8\}\}.$$
The following formula is a consequence of the classical moment-cumulant relations (see e.g. \cite[Proposition 3.2.1]{pt}) together with condition $(\alpha)$ satisfied by the admissible kernel $f$:
\begin{eqnarray}
&&\E[Q_X(f)^4] =  \sum_{i_1,\dots,i_{4d}=1}^n
f(i_1,\ldots,i_d)\ldots f(i_{3d+1},\ldots,i_{4d})\; \mathbb{E}[X_{i_1}\cdots X_{i_{4d}}] \notag\\
&=&\sum_{i_1,\dots,i_{4d}=1}^n
f(i_1,\ldots,i_d)\ldots f(i_{3d+1},\ldots,i_{4d})\;
\sum_{\substack{\sigma \in \mathcal{P}([4d])\\ \sigma \wedge \pi^{\star}= \hat{0}}} \chi_{\sigma}(X_{i_1},\dots,X_{i_{4d}}) \notag\\
&=&
\sum_{m=0}^d \sum_{\bf u} \sum_{\tau } \sum_{j_1,...,j_m=1}^n \sum_{i_1,...,i_{2d-2m}=1}^n\notag\\
 &&\hskip1cm f^{\otimes 4} [j_1,...,j_m, i_1,...,i_{2d-2m}] \, \chi_{\sigma(m, {\bf u}, \tau)} (X_{i_1},\dots,X_{i_{4d}}).\label{e:w}
\end{eqnarray}
In \eqref{e:w}, the following conventions are in order: (i) the second sum runs over all collections of four ordered $m$-tuples with no repetitions $${\bf u} = \{(u_1^1,...,u_m^1), ..., (u_1^4,...,u_m^4)\}$$ such that $\{u_1^j,...,u_m^j\}\subset \{(j-1)d+1,\ldots,jd\}$, $j=1,2,3,4$, (ii) the third sum runs over all $\tau \in \mathcal{P}_2(A({\bf u}))$ respecting the restriction of $\pi^{\star}$ to $A({\bf u})$, and (iii) the factor $$f^{\otimes 4} [j_1,...,j_m, i_1,...,i_{2d-2m}]$$ is obtained along the four steps described below:

\begin{itemize}

\item[(iii-a)] pick an arbitrary enumeration $a_1,...,a_m$ of the $m$ 4-blocks of $\sigma(m, {\bf u}, \tau)$ (by symmetry, the choice of the enumeration of the 4-blocks is immaterial), and let $b_1,...,b_{2d-2m}$ be the enumeration of the 2-blocks of $\sigma(m, {\bf u}, \tau)$ obtained by ordering such blocks according to the size of the smallest element (as described above).

\item[(iii-b)] Consider the mapping in $4d$ variables given by
$$
(x_1,....,x_{4d})\mapsto f^{\otimes 4} (x_1,...,x_{4d}) :=\prod_{j=1}^4 f(x_{(j-1)d+1},...,x_{jd}).
$$

\item[(iii-c)] For $l=1,...m$, set $x_k = j_l$ into the argument of $f^{\otimes 4}$ if and only if $k\in a_l$.

\item[(iii-d)] For $p=1,...,2d-2m$, set $x_q = i_p$ into the argument of $f^{\otimes 4}$ if and only if $q\in b_p$.

\end{itemize}
For instance, for $d=3$ and $(m, {\bf u}, \tau)$ associated with the partition \linebreak$\sigma= \left\{\{1,4,7,10\}, \{2,12\},\{6,8\}, \{5,9\}, \{3,11\}\right\}$ considered in the above example, one has that
$$
f^{\otimes 4} [j, i_1,i_2,i_3,i_4] = f(j,i_1,i_2)f(j,i_3,i_4)f(j,i_4,i_3)f(j,i_2,i_1).
$$

Now, an immediate consequence of the symmetry of $f$ is that, for fixed $m$, the value of the triple sum
$$
\sum_{\tau } \sum_{j_1,...,j_m=1}^n \sum_{i_1,...,i_{2d-2m}=1}^n f^{\otimes 4} [j_1,...,j_m, i_1,...,i_{2d-2m}] \, \chi_{\sigma(m, {\bf u}, \tau)} (X_{i_1},\dots,X_{i_{4d}})
$$
in \eqref{e:w} does not depend on the choice of ${\bf u}$,  and also
 $$\chi_{\sigma(m, {\bf u}, \tau)} (X_{i_1},\dots,X_{i_{4d}})=\chi_4(X)^m,$$
  due to Assumption {\bf (A)}. Since, for every $m$, there are exactly $m!^4\binom{d}{m}^4$ different ways to build a collection ${\bf u}$ of four ordered $m$-tuples with no repetitions such that $$\{u_1^j,...,u_m^j\}\subset \{(j-1)d+1,\ldots,jd\}$$ ($j=1,2,3,4$), we immediately infer the relation
\begin{eqnarray}
\E[Q_X(f)^4]
&=&
\sum_{m=0}^d
\binom{d}{m}^{4}m!^4\chi_4(X)^m \sum_{\tau} \sum_{j_1,...,j_m=1}^n
  \sum_{i_1,...,i_{2d-2m}=1}^n\label{e:w2}\\
&&
\hskip5cm f^{\otimes 4} [j_1,...,j_m, i_1,...,i_{2d-2m}].\notag
\end{eqnarray}

Due to
$$
\chi_{\sigma}(N_{i_1},\dots,N_{i_{4d}}) =
\begin{cases}
\prod\limits_{\{l,j\} \in \sigma} {\bf 1}_{\{i_l= i_j\}}  & \text{ if } \sigma \in \mathcal{P}_2([4d]) \\
0 & \text{ otherwise },
\end{cases}
$$
the summand for $m=0$ equals to $\E[Q_{\mathbf{N}}(f)^4]$, and hence we can rewrite (\ref{e:w2}) as the right-hand side of (\ref{formulaGauss}),  leading to the desired conclusion.
\end{proof}

\subsection{Fourth moment formulae in the non-commutative case.} We now switch to the non-commutative setting considered in Section 2.3.

For a generic non-commutative random variable $Y$, we write $\kappa_j(Y)$, $j=1,2,...$, to indicate the $j$th {\it free cumulant} of $Y$ (see \cite[Lecture 11]{Speicher}).

For $\pi \in \mathcal{NC}([n])$, the generalized \textit{free cumulant} $\kappa_{\pi}(Y)$ is the mapping $Y\mapsto \kappa_\pi(Y)= \prod\limits_{b\in \pi} \kappa_{|b|}(Y)$, verifying the formula
$$ \varphi(Y^n) = \sum_{\pi \in \mathcal{NC}([n])} \kappa_{\pi}(Y),$$
or equivalently
$$\kappa_n(Y) = \sum_{\pi \in \mathcal{NC}([n])}  \mu(\pi, \hat{1}) \prod_{b \in \pi} \varphi(Y^{|b|}),$$
with $\mu(\pi,\hat{1})$ denoting the M\"{o}bius function on the interval $[\pi, \hat{1}]$ (see \cite[Chapter 11]{Speicher} for more details).
We have therefore that $\kappa_1(Y) = \varphi(Y)$, $\kappa_2(Y) = \varphi(Y^2)- \varphi(Y)^2$; if $Y$ is centered, then $\kappa_4(Y) = \varphi(Y^4)-2\varphi(Y^2)^2$. Given a vector of random variables of the type $(Y_{i_1},...,Y_{i_{n}})$ (with possible repetitions), for every partition $\sigma = \{b_1,...,b_k\} \in \mathcal{NC}([n])$, we define the generalized free joint cumulant
$$
\kappa_\sigma(Y_{i_1},...,Y_{i_{n}}) = \prod_{j=1}^k \kappa(Y_{i_a} : a\in b_j),
$$
where $\kappa(Y_{i_a} : a\in b_j)$ is the free joint cumulant of the random variables composing the vector $(Y_{i_a} : a\in b_j)$,
(see \cite[Proposition 11.4]{Speicher}), satisfying the moment-cumulant formula:
$$ \varphi(Y_{i_1}\cdots Y_{i_n}) = \sum_{\sigma \in \mathcal{NC}([n])}\kappa_\sigma(Y_{i_1},...,Y_{i_{n}}).$$
If the $Y_j$'s are freely independent,
$$
\kappa(Y_{i_a} : a\in b_j)=
\begin{cases}
\kappa_{|b_j|} (Y_i) & \text{ if } i_a = i \text{ for every } a\in b_j \\
0 & \text{ otherwise}.
\end{cases}
$$

The following statement contains the non-commutative version of Proposition \ref{propgauss}.

\begin{prop}\label{propgaussfree}
Fix $d\geq 2$ and let $f: [n]^d\to \R$ be an admissible kernel. Then, if $Y$ satisfies Assumption {\bf (B)},
\begin{equation}
\label{formula}
\varphi(Q_\Y(f)^4) = \varphi(Q_\SSw(f)^4) + \kappa_4(Y)\sum_{k=1}^{n}\varphi(Q_\SSw(f(k,\cdot))^4),
\end{equation}
where, for every $k=1,\dots,n$, we have set:
$$ Q_{\SSw}(f(k,\cdot)):= \sum_{k_1,\dots,k_{d-1}=1}^{n} f(k,k_1,\dots,k_{d-1})S_{k_1}\cdots S_{k_{d-1}},$$
with $\SSw=\{S_i : i\geq 1\}$ a sequence of freely independent standard semicircular random variable defined on $(\mathcal{A},\varphi)$.
\end{prop}

\begin{proof}
It is a classic fact that, if ${\bf S} = \{S_i : i\geq 1\}$ is the collection of freely independent semicircular random variables appearing in the statement, then \linebreak$\kappa_\sigma(S_{i_1},...,S_{i_{4d}}) = 0$ whenever $\sigma$ contains a block $b$ such that $|b|\neq 2$; also, when all blocks of $\sigma$ have size 2,
$$
\kappa_{\sigma}(S_{i_1},\dots , S_{i_{4d}}) = \prod_{\{k,l\} \in \sigma} {\bf 1}_{\{i_k=i_l\}}.
$$
From the free moment-cumulant formula, we deduce that
\begin{eqnarray}
&&\varphi\big(Q_{\Y}(f)^4 \big)  \label{e:jeep}\\
 &=&  \sum_{i_1,\dots,i_{4d}=1}^n
f(i_1,\ldots,i_d)\ldots f(i_{3d+1},\ldots,i_{4d})\; \varphi(Y_{i_1}\cdots Y_{i_{4d}}) \notag\\
&=&\sum_{i_1,\dots,i_{4d}=1}^n
f(i_1,\ldots,i_d)\ldots f(i_{3d+1},\ldots,i_{4d})\;
\sum_{\substack{\sigma \in \mathcal{NC}([4d])\\ \sigma \wedge \pi^{\star}= \hat{0}}} \kappa_{\sigma}(Y_{i_1},\dots,Y_{i_{4d}}) \notag.
\end{eqnarray}

In the previous sum, a partition $\sigma \in \mathcal{NC}([4d])$ such that $\sigma \wedge \pi^{\star}= \hat{0}$ gives a non-zero contribution only if its blocks have cardinality either $2$ or $4$. Indeed, since $Y$ is centered, $\sigma$ cannot have any singleton. Similarly, $\sigma$ cannot have any block of cardinality $3$: otherwise, the non-crossing nature of $\sigma$ would imply the existence of at least one singleton and the corresponding cumulant would vanish. Therefore, the only non-vanishing terms in the sum appearing in the last line of \eqref{e:jeep} are those corresponding to non-crossing partitions that respect $\pi^{\star}$, and whose blocks have cardinality either $2$ or $4$.  Recall that this class of partitions is denoted by $\mathcal{NC}_{2,4}^{\star}([4d])$,
in such a way that we can write
\begin{eqnarray}
&&\varphi\big(Q_{\Y}(f)^4 \big)  \label{e:jh}\\
 &=&\sum_{i_1,\dots,i_{4d}=1}^n
f(i_1,\ldots,i_d)\ldots f(i_{3d+1},\ldots,i_{4d})\;
\sum_{\sigma \in \mathcal{NC}_{2,4}^{\star}([4d])} \kappa_{\sigma}(Y_{i_1},\dots,Y_{i_{4d}}) \notag.
\end{eqnarray}
In view of the non-crossing nature of the involved partitions, one can indeed give a very precise description of the class $\mathcal{NC}_{2,4}^{\star}([4d])$ as the following disjoint union:
$$
\mathcal{NC}_{2,4}^{\star}([4d]) = \mathcal{NC}^\star_2([4d]) \cup \{\rho_1,...,\rho_d\}
$$
where, for $h=1,...,d$, the partition $\rho_h$ contains exactly one block of size $4$, given by
$$
\{ h, 2d-h+1, 2d+h, 4d-h+1\}
$$
and the remaining blocks of size 2 are completely determined by the non-crossing nature of $\rho_h$:
\begin{enumerate}
\item $j \sim 2d-j+1$, for $j=h+2,\dots,d$;
\item $j \sim 4d-j+1$, for $j=1,\dots,h$;
\item $2d + j \sim 2d-j+1$, for $j=1,\dots,h$;
\item $ 2d+j \sim 4d-j+1$, for $j=h+2,\dots,d$.
\end{enumerate}
Now, if $\sigma\in \mathcal{NC}^\star_2([4d])$, the freely independence of the $Y_i$'s yields that
$$
\kappa_{\sigma}(Y_{i_1},\dots,Y_{i_{4d}}) = \prod_{\{k,l\} \in \sigma} {\bf 1}_{\{i_k=i_l\}}= \kappa_{\sigma}(S_{i_1},\dots , S_{i_{4d}}),
$$
in such a way that
\begin{eqnarray*}
&& \sum_{i_1,\dots,i_{4d}=1}^n
f(i_1,\ldots,i_d)\ldots f(i_{3d+1},\ldots,i_{4d})\;
\sum_{\sigma \in \mathcal{NC}_{2}^{\star}([4d])} \kappa_{\sigma}(Y_{i_1},\dots,Y_{i_{4d}})
\\&&= \sum_{i_1,\dots,i_{4d}=1}^n
f(i_1,\ldots,i_d)\ldots f(i_{3d+1},\ldots,i_{4d})\;
\sum_{\sigma \in \mathcal{NC}_{2,4}^{\star}([4d])} \kappa_{\sigma}(S_{i_1},\dots,S_{i_{4d}})
\\  &&= \varphi(Q_{\bf S}(f)^4)
\end{eqnarray*}
(it is indeed an easy exercise to show that the last equality is exactly equivalent to \eqref{e:knps}). On the other hand, for $h=1,...,d$, one has that
\begin{eqnarray*}
&&\kappa_{\rho_h}(Y_{i_1},\dots,Y_{i_{4d}}) = \kappa_4(Y) {\bf 1}_{\{i_h = i_{2d-h+1} = i_{2d+h} = i_{4d-h+1}\}} \prod_{\{k,l\}\in \sigma} {\bf 1}_{\{i_k = i_l\}}.
\end{eqnarray*}
The symmetry of $f$ yields the following identities:
\begin{eqnarray*}
&& \sum_{i_1,\dots,i_{4d}=1}^n
f(i_1,\ldots,i_d)\ldots f(i_{3d+1},\ldots,i_{4d})
 \kappa_{\rho_1}(Y_{i_1},\dots,Y_{i_{4d}})\\&& =  \sum_{i_1,\dots,i_{4d}=1}^n
f(i_1,\ldots,i_d)\ldots f(i_{3d+1},\ldots,i_{4d})
 \kappa_{\rho_d}(Y_{i_1},\dots,Y_{i_{4d}}) \\
 && = \kappa_4(Y)\sum_{k=1}^n \left( \sum_{j_1,...,j_{d-1}=1}^nf(k,j_1,...,j_{d-1})^2  \right)^2,
\end{eqnarray*}
and, for $h=2,...,d-1$, writing $s=d-h$,
\begin{eqnarray*}
&&\sum_{i_1,\dots,i_{4d}=1}^n
f(i_1,\ldots,i_d)\ldots f(i_{3d+1},\ldots,i_{4d})
 \kappa_{\rho_{h}}(Y_{i_1},\dots,Y_{i_{4d}})\\
&&= \kappa_4(Y)\sum_{k=1}^n \sum_{j_1,...,j_{d-1-s}=1}^n \sum_{k_1,...,k_{d-1-s}=1}^n \\
&&\quad\left(\sum_{a_1,...,a_{s}=1}^n f(k, j_1,...,j_{d-s-1}, a_1,...,a_{s})f(k, k_1,...,k_{d-1-s}, a_1,...,a_{s})  \right)^2.
\end{eqnarray*}
Summing up the previous expression for all $s=1,...,d-2$, and for $h=1,d$, a straightforward application of formula \eqref{e:knps} in the case $m=d-1$ and $g = f(k, \cdot)$ therefore implies that
\begin{eqnarray*}
&&\sum_{h=1}^d \sum_{i_1,\dots,i_{4d}=1}^n
f(i_1,\ldots,i_d)\ldots f(i_{3d+1},\ldots,i_{4d})
 \kappa_{\rho_{h}}(Y_{i_1},\dots,Y_{i_{4d}})\\
 && =\kappa_4(Y)\sum_{k=1}^{n}\varphi(Q_\SSw(f(k,\cdot))^4),
\end{eqnarray*}
from which the desired conclusion follows.
\end{proof}

\section{Proofs}

\subsection{Proof of Theorem \ref{superTeo1}}

We shall show that, for a given sequence of kernels $\{f_n : n\geq 1\}$, the implications (iii)$\to$(ii)$\to$(i)$\to$(iii) are in order, with (i), (ii) and (iii) as given in Definition \ref{defcom}.

{\it Proof of }(iii)$\to$(ii).
Consider the formula (\ref{formulaGauss}), and bear in mind that $\chi_4(X) \geq 0$ due to our hypothesis. Let $N\sim \mathcal{N}(0,1)$ and let $\NN=\{N_i\}_{i\geq 1}$  be the corresponding sequence of independent copies.  Moreover, $\chi_4(Q_{\NN}(f_n))=\E[Q_{\NN}(f_n)^4] -3$ is positive (see indeed \cite[Lemma 5.2.4]{NPbook}). On the other hand,
formula \eqref{formulaGauss} entails that $$
\chi_4(Q_{\X}(f_n))=\E[Q_{\X}(f_n)^4] - 3  \geq  \E[Q_{\NN}(f_n)^4] - 3,
$$
from which one deduces that, if (iii) holds, then $\E[Q_{\NN}(f_n)^4]  \rightarrow 3$.
Since $N$ satisfies the ${CFMT}_d$ (Theorem \ref{np}), it follows that $(ii)$ takes place.

{\it Proof of} (ii)$\to$(i). It is the conclusion of Theorem \ref{inv}.

{\it Proof of} (i)$\to$(iii).  Due to \cite[Lemma 4.2]{NourdinPeccatiReinert}, the sequence $Q_{\X}(f_n)^4$ is uniformly integrable. Therefore, the conclusion follows as for the proof of \cite[Theorem 11.3.1]{NPbook}.
\qed

\begin{rmk}
We have chosen to deal with the i.i.d. case just to ease the notation and the discussion. We could have considered a sequence $\X=\{X_i\}_{i \geq 1}$ of independent centered random variables, with unit variance, possibly not identically distributed, but starting from an inequality instead of an equality. Indeed, for every $m=1,\dots,d$ and every $n$, set $\gamma_{n}^{(m)} = \min\limits_{i_1,\dots,i_m \in [n]} \prod\limits_{l=1}^m \chi_4(X_{i_l})$, and $A_n = \min\limits_{m=1,\dots,d}\gamma_n^{(m)}$.
Assume further that there exists $A \geq 0$ such that $\inf\limits_{n \geq 1}A_n \geq A$ (which is clearly satisfied if $\chi_4(X_i) \geq 0$ for all $i$). Then, from the inequality
\small{
\begin{align*}
&\E[Q_{\X}(f)^4] \geq \E[Q_{\mathbf{N}}(f)^4] + \\
&+ A \sum_{m=1}^d \binom{d}{m}^4  m!^4  \sum_{\bf u} \sum_{\tau } \sum_{j_1,...,j_m=1}^n \sum_{i_1,...,i_{2d-2m}=1}^n  f^{\otimes 4} [j_1,...,j_m, i_1,...,i_{2d-2m}]\; ,
\end{align*}
}\normalsize
\noindent it is possible to handle the case when the $X_i$'s are independent centered r.v.'s, with unit variance, but not necessarily identically distributed, yielding that  $\E[Q_{\X}(f_n)^4]  \rightarrow 3$ is a sufficient condition for the convergence $Q_{\X}(f_n) \stackrel{\text{Law}}{\rightarrow} \mathcal{N}(0,1)$.
\end{rmk}

\begin{rmk}
An extension of Theorem \ref{superTeo1} in the setting of Gamma approximation of homogeneous sums of even degree $d\geq 2$ can be achieved in the following way. If $\nu >0$, let $G(\frac{\nu}{2})$ denote a random variable with Gamma distribution of parameter $\frac{\nu}{2}$, and set $F(\nu) \stackrel{\text{law}}{=} 2G(\frac{\nu}{2})- \nu$. If $\E[Q_{\mathbf{N}}(f_n)^2] \rightarrow 2\nu$ and $\E[Q_{\mathbf{N}}(f_n)^3] \rightarrow 8\nu$, from  identity (3.7) in \cite{NourdinPeccati2} it follows that
$$ \E[Q_{\mathbf{N}}(f_n)^4] -12 \E[Q_{\mathbf{N}}(f_n)^3]  - (12 \nu^2 - 48\nu) > 0,$$
for sufficiently large $n$.
Then, under the assumption $\E[X^3] = 0$, $\E[Q_{\X}(f_n)^3] = \E[Q_{\mathbf{N}}(f_n)^3]$, one has that:
\small{
\begin{align*}
\E[Q_{\mathbf{X}}(f_n)^4] & -12\E[Q_{\X}(f_n)^3] - (12 \nu^2 - 48\nu) = \\
& \E[Q_{\mathbf{N}}(f_n)^4]-12\E[Q_{\mathbf{N}}(f_n)^3] - (12 \nu^2 - 48\nu) \\
&+ \sum_{m=1}^d \binom{m}{d}^{4}m!^4 \chi_4(X)^m \sum_{j_1,\dots,j_m=1}^n \E[Q_{\mathbf{N}}(f_n(j_1,\dots,j_m,\cdot))^4],
\end{align*}
}
\normalsize
By exploiting \cite[Theorem 1.2]{NourdinPeccati2} and with the same strategy as in the proof of Theorem \ref{superTeo1}, it is possible to provide a Fourth Moment statement for $Q_{\mathbf{X}}(f_n)$ when the target is the the Gamma distribution (note that the universality of Gaussian homogeneous sums w.r.t. Gamma approximation has been established in \cite[Theorem 1.12]{NourdinPeccatiReinert} both for homogeneous sums with i.i.d. entries and with only independent entries).
\end{rmk}

\subsection{Proof of Theorem \ref{superTeo2}}

We shall show the series of implications \linebreak (iii)'$\to$(ii)'$\to$(i)'$\to$(iii)', with (i)', (ii)' and (iii)' as given in Definition \ref{defnoncom}.

{\it Proof of }(iii)'$\to$(ii)'.
Consider the formula (\ref{formula}) and bear in mind that $\kappa_4(Y) \geq 0$ due to our hypothesis. Let $S\sim \mathcal{S}(0,1)$ and let ${\bf S}=\{S_i\}_{i\geq 1}$
 be the corresponding sequence of freely independent copies.
We then have that $\kappa_4(Q_{\bf S}(f_n))=\varphi(Q_{\bf S}(f_n)^4) -2$ is positive,
see indeed \cite[Corollary 1.7]{NourdinPeccatiSpeicher}. On the other hand,
formula \eqref{formula} entails that
$$
d!^2 \kappa_4(Q_{\Y}(f_n))= d!^2 \varphi(Q_{\Y}(f_n)^4) - 2  \geq  d!^2\varphi(Q_{\bf S}(f_n)^4) - 2 >  \varphi(Q_{\bf S}(f_n)^4) - 2
$$
from which one deduces that, if (iii)' holds true, then $\varphi(Q_{\bf S}(f_n)^4)  \rightarrow 2$.
Since $S\sim \mathcal{S}(0,1)$ satisfies the ${NCFMT}_d$ (Theorem \ref{knps}), one has that (ii)' takes place.

{\it Proof of }(ii)'$\to$(i)'. It is the conclusion of Theorem \ref{invnoncom}.

{\it Proof of }(i)'$\to$(iii)'. It comes from the very definition of the convergence in law in the free case (that is, the convergence of all the moments).
\qed

\begin{rmk}
As in the commutative case, we have chosen to deal with the i.i.d. case just to ease the notation and the discussion. We could have considered a sequence $\Y=\{Y_i\}_{i \geq 1}$ of freely independent centered random variables, with unit variance, possibly not identically distributed, but starting from an inequality other than an equality. Indeed, for every $n\geq 1$, set $\beta_n = \min\limits_{i=1,\dots,n}\kappa_4(Y_i)$, and assume that there exists $\beta \geq 0$ such that $\inf\limits_{n \geq 1}\beta_n \geq \beta$. In particular, $\lim\limits_{n \rightarrow \infty}\beta_n = \inf\limits_{n \geq 1}\beta_n \geq \beta$. Then,
$$ \varphi\big( Q_{\mathbf{Y}}(f)^4 \big)  \geq  \varphi\big( Q_{\mathbf{S}}(f)^4 \big) + \beta \sum_{k=1}^n \varphi\big(Q_{\mathbf{S}}(f(k,\cdot))^4 \big),
$$
entailing that the sequence of freely independent random variables $\Y= \{Y_i\}_{i\geq 1}$ satisfies the Fourth Moment Theorem for homogeneous sums in dimension $d$ for semicircular approximation: that is, $d!^2 \kappa_4(Q_{\mathbf{Y}}(f_n)) \rightarrow 0$ is a necessary and sufficient condition for the convergence in law of $Q_{\Y}(f_n)$ to the semicircle law.
\end{rmk}

\begin{rmk}
Assume that $d \geq  2$ is even. Then, $\varphi\big(Q_\Y(f)^3\big) = \varphi\big(Q_\SSw(f)^3\big)$ (indeed, if $\mathcal{NC}_{>0}$ denotes the set of partitions with no singleton, $|\mathcal{NC}_{>0}^{\star}([3d])| = |\mathcal{NC}_2^{\star}([3d])|$) and so, if further we assume that $\varphi\big(Q_\Y(f)^2\big) = \varphi\big(Q_\SSw(f)^2\big) = \lambda> 0$, from (\ref{formula}) it follows that:
\small
\begin{align*}
\varphi(Q_\Y(f)^4)-2\varphi(Q_\Y(f)^3) - (2\lambda^2- \lambda) &= (\varphi(Q_\SSw(f)^4)-2\varphi(Q_\SSw(f)^3) -(2\lambda^2- \lambda)) \\
&+ \kappa_4(Y)\sum_{k=1}^{n}\varphi(Q_\SSw(k,\cdot)^4).
\end{align*}
\normalsize
From this formula, considering the Fourth Moment Theorem for free Poisson approximation of Wigner chaos established in \cite{NourdinPeccati},  the analogous of the Theorem \ref{superTeo2} with respect to the Free Poisson limit can be achieved.
\end{rmk}

\section{On the existence of thresholds for the fourth moment criterion}

In Section 2, we have shown that the condition {\bf(A)} together with $\E[X^4]\ge 3$ (resp. {\bf(B)} together with $\varphi(Y^4)\ge 2$) is sufficient to ensure that $X$ is $U_d$ and satisfies the $CFMT_d$ (resp. $Y$ is $FU_d$ and satisfies the $NCFMT_d$).

It is a natural question to determine whether this condition is necessary. In other words, for a given degree $d\geq 2$, what is the smallest real number $r_d\in (1,3]$ (resp. $s_d \in (1,2]$) such that, if $\E[X^4]$  (resp. $\varphi(Y^4)$) is greater than $r_d$ (resp. $s_d$), then
$X$ is $U_d$ and satisfies the $CFMT_d$ (resp. $Y$ is $FU_d$ and satisfies the $NCFMT_d$)?

\subsection{Existence of a threshold in the classical case}\label{nor}

So far, in the classical case we are able to determine the existence of the threshold only under the extra assumption that $X$ is $U_d$. Before stating our results, we need some preliminary results.

\begin{prop}\label{prop51}
Let $X$ satisfy {\bf (A)}, $U_d$ and $CFMT_d$, for $d\geq 2$. Then either $\chi_4(Q_\X(f)) < 0$ for every admissible kernel $f$, or $\chi_4(Q_\X(f)) > 0$ for every admissible kernel $f$.
\end{prop}

\begin{proof}
\underline{Step 1:} First, we prove that if $X$ is $CFMT_d$ and $U_d$, then \linebreak $\chi_4(Q_\X(f)) \neq 0$ for every admissible kernel $f$. Indeed, if there exists $f$ such that $\chi_4(Q_{\X}(f)) = 0,$ then the constant sequence $Q_{\X}(f)$ will be normal, and then, being $X$ $U_d$, we would have $Q_{\mathbf{N}}(f) \stackrel{\text{law}}{=} \mathcal{N}(0,1)$, which is absurd. \\

\underline{Step 2:} Assume now that there exist two admissible kernels $f_{0}$ and $f_{1}$ such that $\E[Q_{\mathbf{X}}(f_{0})^4]>3$ and $\E[Q_{\mathbf{X}}(f_{1})^4]<3$. Consider, for every $t \in [0,1]$, the admissible kernel
$$
f_t = \dfrac{ t f_1+ (1-t)f_0}{\sqrt{\E[(t Q_{\mathbf{X}}(f_{1}) + (1-t)Q_{\mathbf{X}}(f_{0}))^2] }}.
$$
Since $\chi_4\big(Q_{\mathbf{X}}(f_{1})\big) < 0$ and $\chi_4\big(Q_{\mathbf{X}}(f_{0})\big) > 0$, there exists $t^{\star} \in (0,1)$ such that $\chi_4\big(Q_{\mathbf{X}}(f_{t^\star})\big) = 0$, which is impossible due to the conclusion of the first step.
\end{proof}

\begin{rmk}
It is always possible to construct a homogeneous sum with positive fourth cumulant. Indeed, for $n$ large enough, consider the homogeneous polynomial (with admissible kernel $g_n$)
$$Q_{\mathbf{X}}(g_n) =
X_1\times\dfrac{X_2^{(1)} \cdots X_d^{(1)} + \cdots + X_2^{(n-1)} \cdots X_d^{(n-1)}}{\sqrt{n-1}} ,$$ where $X_{j}^{(i)}$ is a sequence of independent copies of $X$. A direct computation provides
$$
\E[Q_{\mathbf{X}}(g_n)^4]=\E[X^4]\left(3+\dfrac{\E[X^4]^{d-1}-3}{n-1}\right) \to 3\E[X^4] > 3.
$$
\end{rmk}

\begin{prop}
Let $X$ satisfy {\bf (A)}, $CFMT_d$ and $U_d$. Then necessarily $\E[X^4] > \sqrt[d]{3}$.
\end{prop}

\begin{proof}
As a consequence of the previous proposition, if $X$ is $CFMT_d$ and $U_d$, then $\E[X^4] \neq \sqrt[d]{3}$ (otherwise, for $Q_\X(f) = X_1 \cdots X_d$ we would have $\chi_4(Q_{\X}(f))=0$).
Now we will proceed by contradiction: assume that $\E[X^4] \in (1,\sqrt[d]{3})$. Then, for $Q_\X(f) = X_1 \cdots X_d$, $\E[Q_{\X}(f)^4] < 3$, and hence we expect $\E[Q_{\X}(g)^4] < 3$ for any other admissible kernel $g$, which contradicts the previous remark.
 Hence if $\E[X^4] < \sqrt[d]{3}$, $X$ cannot satisfy the $CFMT_d$.
\end{proof}

The main result of this section is stated in the next theorem.

\begin{thm}
\label{ThresholdClassic}
Let $d\geq 2$. Then, there exists a real number $r_d\in(\sqrt[d]{3},3]$ such that, for any random variable $X$ satisfying {\bf(A)} and being $U_d$, the following are equivalent:
\begin{enumerate}
\item $X$ satisfies the $CFMT_d$;
\item $\E[X^4] \geq r_d$.
\end{enumerate}

\end{thm}

\begin{proof}
Let $X$ satisfy the $CFMT_d$ and $U_d$. Then, as a result of the above discussion, $\E[Q_{\X}(f)^4] > 3$ for every admissible kernel $f$. Let $Z$ be a centered random variable, with $\E[Z^3]=0$ and unit variance, such that $\E[Z^4] \geq \E[X^4]$. To obtain the existence of the desired threshold $r_d$, it is enough to show that $Z$ satisfies the $CFMT_d$ as well. We will proceed in several steps, considering mixtures between $X$ and a suitable random variable $T$.\\

\underline{Step\;1}.
Set $\theta=\E[Z^4]/\E[X^4]-1$, let $q\in\mathbb{N}$, and
set $\alpha=\sqrt{(1+\theta)^{1/q}-1}$. Since
$(1+\theta)^{1/q}\to 1$ as $q\to\infty$, one may and will choose $q$ large enough so that $\alpha\in [0,1)$. Now, let $V_1,\ldots,V_q$ be independent copies with distribution $\frac{1}{2}(\delta_{1-\alpha} + \delta_{1+\alpha})$. Assume further that $V_1,\ldots,V_q$ and $X$ are independent. Then, the random variable
$T=\sqrt{V_1\ldots V_q}$ is independent of $X$, takes its values in $[x,\infty[$ with $x=(1-\alpha)^{q/2}>0$, and satisfies $\E[T^2]=1$ and $\E[T^4]=
(1+\alpha^2)^q=
\E[Z^4]/\E[X^4]$.

Let $X_1,X_2,\ldots$ (resp. $T_1,T_2,\ldots$ and $Z_1,Z_2,\ldots$) be a sequence of independent copies of $X$ (resp. $T$ and $Z$).\\

\underline{Step\;2}.
Set $Q_{\TX}(f_n)=\sum\limits_{i_1,\cdots,i_d=1}^nf_n(i_1,\cdots,i_d)(T_{i_1}X_{i_1})\cdots(T_{i_d}X_{i_d})$. We deduce from Step 1 that $\E[Q_{\TX}(f_n)^2]=1$ and
$\E[Q_{\TX}(f_n)^4]=\E[Q_{\mathbf{Z}}(f_n)^4]$.
Separating the expectation according to the sequences $\{T_i\}$ and $\{X_i\}$ (setting $\E_\T[\cdot]=\E[\cdot|\X]$ and $\E_\X[\cdot] = \E[\cdot|\T]$), we can write:
\begin{eqnarray}
&& \E[Q_{\mathbf{Z}}(f_n)^4]- 3 =\notag \\
&=& \E_\T\left[\E_\X\left[Q_{\TX}(f_n)^4\right]-3\right]  \notag\\ 
&=&\E_\T \left[\E_\X\left[Q_{\TX}(f_n)^4\right] - 6\E_\X\left[Q_{\TX}(f_n)^2\right]+3\right]\notag \\ 
&=& \E_\T\left[\left(\E_\X\left[Q_{\TX}(f_n)^4\right]-3\E_\X\left[Q_{\TX}(f_n)^2\right]^2\right)+3\left(\E_\X\left[Q_{\TX}(f_n)^2\right]-1\right)^2\right].\notag\\
&&\label{eq-T-N-2}
\end{eqnarray}

Since $\E[X^4] > \sqrt[d]{3}$ and $X$ is $U_d$ and satisfies the $CFMT_d$, almost surely in $\{T_i\}_{i\ge 1}$ one has, according to Proposition \ref{prop51}, that:
$$\chi_4(\E_{\mathbf{X}}[Q_{\TX}(f_n)]) = \E_\X\left[Q_{\TX}(f_n)^4\right]-3\E_\X\left[Q_{\TX}(f_n)^2\right]^2\ge 0.$$

Assume now that $\E[Q_{\mathbf{Z}}(f_n)^4]\to 3$ as $n \rightarrow \infty$, that is,
\[
\E_\T\left[ \left(\E_\X\left[Q_{\TX}(f_n)^4\right]-3\E_\X\left[Q_{\TX}(f_n)^2\right]^2\right)+3\left(\E_\X\left[Q_{\TX}(f_n)^2\right]-1\right)^2       \right ] \to 0 ,
\]
see indeed (\ref{eq-T-N-2}).
Up to extracting a subsequence and due to the positivity of summands, we deduce that, almost surely in $\{T_i\}_{i\ge1}$,
\begin{eqnarray}
\E_\X\left[Q_{\TX}(f_n)^2\right]&\to& 1,\notag\\
\E_\X\left[Q_{\TX}(f_n)^4\right]&\to& 3.\label{ymca}
\end{eqnarray}

\underline{Step\; 3}.
Since $X$  satisfies the $CFMT_d$,
we deduce from (\ref{ymca}) that, almost surely in $\{T_i\}$,
\begin{eqnarray*}
Q_{\TX}(f_n)&\xrightarrow[n\to\infty]{\text{Law}}&\mathcal{N}(0,1).
\end{eqnarray*}
But $X$ is assumed to be $U_d$ so, by Theorem \ref{inv}, it follows
that, almost surely in $\{T_i\}_{i\ge1}$,
\begin{eqnarray*}
\max_{1\leq i_1\leq n}\sum_{i_2,\cdots,i_d=1}^n f_n(i_1,i_2,\cdots,i_d)^2 (T_{i_1}\cdots T_{i_d})^2 &\xrightarrow[n\to\infty]{\text{}}&0.
\end{eqnarray*}
Using that $T_i\geq x>0$ for all $i$, we finally deduce that
$$
\max_{1\leq i_1\leq n}\sum_{i_2,\cdots,i_d=1}^n f_n(i_1,\cdots,i_d)^2  \xrightarrow[n\to\infty]{\text{}}0.
$$
This condition and de Jong's criterion (see \cite[Theorem 1.9]{NourdinPeccati4} for a modern proof),
ensure that $Z$ satisfies the $CFMT_d$ as well. Besides, the law of $Z$ is $U_d$.\\

\underline{Step 4} (Conclusion). The desired threshold $r_d$ is thus given as being
the smallest real number $t\in (\sqrt[d]{3},3]$ such that there exists $X$ satisfying {\bf(A)}, the $CFMT_d$, being $U_d$ and with $\E[X^4]=t$.
\end{proof}

\begin{rmk}
In \cite[Proposition 4.6]{NourdinPeccatiReinert}, the authors provided the $CFMT_2$ when $\E[X^4]=1$, that is, when we are dealing with the Rademacher chaos of order 2. The case $d\geq 3$ is still open to the best of our knowledge. Neverthless, the reader should bear in mind that Rademacher chaos is not $U_d$ (see, for instance, \cite{NourdinPeccatiReinert}), and hence such result is not in contrast with our approach.\\
\end{rmk}

\subsection{Existence of a threshold in the free case}

In order to generalize  in the free probability setting the technique of the mixtures used in Section \ref{nor}, one should consider a sequence $\{Z_i\}$ of freely independent random variables, freely independent of $\{S_i\}$ in such a way that the $Z_i$'s and the $S_j$'s commute. But this is possible only if $Z_i$ has vanishing variance (see \cite[Lecture 5]{Speicher}).
Hence we will adopt a different approach to reach the free counterpart of Theorem \ref{superTeo1}, which, in the meantime, will highlight the simplicity and usefulness  of formula \eqref{formula}, that avoid us to assume  that $Y$ is $FU_d$. Let us remark that since formula \eqref{formulaGauss} is not linear in $\chi_4(X)$, the following strategy cannot be adapted for the determination of the threshold in the classical case.

\begin{rmk}
It is always possible to construct a sequence of homogeneous sum $Q_{\bf Y}(g_n)$ converging in distribution towards an element of the $d$-th Wigner chaos, and hence having positive fourth cumulant. Indeed, if $\mathbf{Y}=\{Y_i\}_{i\geq 1}$ is a sequence of freely independent and identically distributed random variables, centered, and with unit variance, for every $n \in \mathbb{N}$ and every $i=1,\dots,n$, set:
$$ Z_{n}^{(i)} = \dfrac{1}{\sqrt{n}}\sum_{j=1}^n Y_{(j-1)d + i}. $$
Then, consider the homogeneous sum $ Q_{\Y}(g_n) = d!^{-1} \sum\limits_{ \sigma \in \Sigma_d } Z_n^{(\sigma(1))}\cdots Z_n^{(\sigma(d))}.$ First, the free CLT states that $\{Z_n^{(i)}\}_n \stackrel{\text{law}}{\rightarrow} S^{(i)} \sim \mathcal{S}(0,1)$ for every $i=1,\dots,d$, with the $S^{(i)}$'s freely independent. Then, the multidimensional CLT (see \cite[Theorem 8.17]{Speicher}) assures that $Q_{\Y}(g_n)$ converges to an element in the $d$-th chaos of Wigner. Since any element of the $d$-th Wigner chaos has a stricly positive fourth cumulant (see \cite[Corollary 1.7]{NourdinPeccatiSpeicher}), we get that $\kappa_4(Q_{\Y}(g_n))>0$ for all $n$ large enough.
\end{rmk}

\begin{thm}
\label{ThresholdFree}
Let $d\geq 2$. There exists a real number $s_d\in (1,2]$ such that, for any random variable $Y$ satisfying {\bf(B)},
\begin{enumerate}
\item[(i)] If $\varphi(Y^4) \geq s_d$, then $Y$ satisfies the $NCFMT_d$;
\item[(ii)] If $1<\varphi(Y^4)<s_d$, then $Y$ does not satisfy the $NCFMT_d$.
\end{enumerate}
\end{thm}

\begin{proof}

The proof is divided into three steps.

\underline{Step \; 1}. Let $t\in ]1,2]$ be a real number such that, for all $Y\in\mathcal{A}$ centered with unit variance, we have the following implication:
\begin{equation}\label{H}
\Big(\varphi(Y^4)=t\Big) ~~~\Rightarrow~~~~ \text{$Y$ satisfies the $NCFMT_d$}.
\end{equation}
We will show that, under this assumption, $\kappa_4(Q_{\Y}(f)) > 0$ for every $f$ or $\kappa_4(Q_{\Y}(f)) < 0$ for every $f$. By contradiction, assume that there exist two admissible kernels $f,g:[n]^d\to\R$ such that $\kappa_4(Q_{\Y}(f))<0$ and $\kappa_4(Q_{\Y}(g))>0$. By taking an appropriate linear combination of $f$ and $g$, we can construct an admissible kernel $h$ such that $\kappa_4(Q_{\Y}(h))=0.$
Using Proposition \ref{propgaussfree}, for any free random variable $Z$ centered with unit-variance such that $\ph(Z^4)=\ph(Y^4)=t$, we also have $\kappa_4(Q_{\Z}(h))=0.$ But $Z$ satisfies the $NCFMT_d$ by (\ref{H}), so the constant sequence $Q_{\Z}(h)$ is semicircular. Summarizing, under the assumption (\ref{H}), we have that $Q_{\Z}(h)$ is semicircular (implying that $\kappa_p(Q_{\Z}(h))=0$ for any $p\ge 3$) for all centered $Z$ with unit variance such that $\ph(Z^4)=t$.
On the other hand, let $E$ denote the set
$$
\Big\{(\ph(Z^3),\ph(Z^5),\ph(Z^6))\in\R^3\Big|Z\in \mathcal{A},\,\ph(Z)=0,\ph(Z^2)=1,\ph(Z^4)=t\Big\}.
$$
From above, by expanding $\kappa_6(Q_{\Z}(h))=0$ as a multivariate polynomial $P_t$ in $\ph(Z^3),\ph(Z^5),\ph(Z^6)$, it turns out that $E\subset \{(a,b,c)|P_t(a,b,c)=0\}$. In particular, $E$ has zero Lebesgue measure. On the other hand, since the criterion of solvability of the Hamburger moment problem is a necessary condition for the solvability of the Hausdorff moment problem, if  $(a,b,c)\in E$, then in particular the Hankel matrix $M_t=(\ph(Z^{i+j}))_{0\le i,j \le 3}$ is positive definite (see \cite[Theorem 6.1]{Chihara}).
However, the set of triplets $(a,b,c)$ such that $M_t$ is positive is a non-empty open subset of $\R^3$ and has then a positive Lebesgue measure. By contradiction, we deduce that the existence of the admissible kernel $h$ is impossible, meaning that either $\kappa_4(Q_Y(f))> 0$ for all admissible $f:[n]^d\to\R$, or $\kappa_4(Q_Y(f))< 0$ for all admissible $f:[n]^d\to\R$. By virtue of the previous remark, if $Y$ satisfies the $NCFMT_d$, \linebreak then $\kappa_4(Q_Y(f))> 0$ for every admissible kernel $f$.\\

\underline{Step\; 2}:
Let $Z$ be a centered random variable with unit variance such that $\ph(Z^4)>t$ (with $t$ such as (\ref{H})). We shall prove that $Z$ satisfies the $NCFMT_d$. So, assume that $\kappa_4(Q_{\Z}(f_n))\to 0$ as $n\to\infty$ for a given sequence of admissible kernels $f_n$. By applying the formula (\ref{formula}) to $Y$ and $Z$ and by taking the difference, we obtain:
\begin{equation}
\label{equazione}
d!^2\kappa_4(Q_{\Z}(f_n))=d!^2\kappa_4(Q_{\Y}(f_n))+(\ph(Z^4)-t)d!^2\sum_{k=1}^{n} \varphi\big(Q_\SSw(f_n(k,\cdot))^4\big),
\end{equation}
from which it follows that $\sum\limits_{k=1}^n\varphi\big(Q_\SSw(f_n(k,\cdot))^4\big) \to 0$.
Therefore, considering in the limit the formula (\ref{formula}) for $Z$, we infer that \linebreak
$\kappa_4(Q_{\SSw}(f_n)) \to 0$ as $n\to\infty$.
Theorem \ref{knps} combined with Theorem \ref{invnoncom} allows to conclude that $\sqrt{d!} Q_{\bf Z}(f_n) \stackrel{\text{Law}}{\to} \mathcal{S}(0,1)$ as $n\to\infty$. That is, $Z$ satisfies the $NCFMT_d$.\\

\underline{Step\; 3} (Conclusion). The desired threshold $s_d$ is thus given as being
the smallest real number $t\in ]1,2]$ such that there exists $Y$ satisfying {\bf (B)}, the $NCFMT_d$ and
 with $\varphi(Y^4)=t$.
\end{proof}

\end{document}